\documentclass{amsart}
\usepackage{amssymb,amsmath,amsfonts}
\usepackage[all]{xy}
\usepackage{graphics}
\usepackage{color}
\usepackage{tikz}
\usepackage{epsfig}
\usepackage[english]{babel}

\title[BV Calabi--Yau threefolds]{Borcea--Voisin Calabi--Yau threefolds and invertible potentials}

\author{Michela Artebani}

\address{Departamento de Matem\'atica, Universidad de Concepci\'on, Casilla 160-C, Concepci\'on, Chile}

\email{martebani@udec.cl}

\author{Samuel Boissi\`ere}

\address{Laboratoire de Math\'ematiques et Applications,
			Universit\'e de Poitiers, T\'el\'eport 2, Boulevard Marie et Pierre Curie,
			F-86962 Futuroscope Chasseneuil}

\email{Samuel.Boissiere@math.univ-poitiers.fr}

\urladdr{http://www-math.sp2mi.univ-poitiers.fr/$\sim$sboissie/}

\author{Alessandra Sarti}

\address{Laboratoire de Math\'ematiques et Applications, 
			Universit\'e de Poitiers, T\'el\'eport 2, Boulevard Marie et Pierre Curie,
			F-86962 Futuroscope Chasseneuil}
			
\email{sarti@math.univ-poitiers.fr}

\urladdr{http://www-math.sp2mi.univ-poitiers.fr/$\sim$sarti/}

\date{\today}
\dedicatory{with an appendix in collaboration with Sara A. Filippini}

\newtheorem{theorem}{Theorem}[section]
\newtheorem{prop}[theorem]{Proposition}
\newtheorem{lemma}[theorem]{Lemma}

\theoremstyle{definition}

\newtheorem{remark}[theorem]{Remark}
\newtheorem{example}[theorem]{Example}

\DeclareMathOperator{\Pic}{Pic}

\DeclareMathOperator{\SL}{SL}

\DeclareMathOperator{\diag}{diag}
\DeclareMathOperator{\id}{id}

\DeclareMathOperator{\rk}{rk}
\DeclareMathOperator{\Aut}{Aut}

\newcommand{\IC}{\mathbb{C}}
\newcommand{\IQ}{\mathbb{Q}}
\newcommand{\IZ}{\mathbb{Z}}

\newcommand{\IP}{\mathbb{P}}

\newcommand{\cK}{\mathcal{K}}

\newcommand{\cZ}{\mathcal{Z}}

\newcommand{\lra}{\longrightarrow}

\newcommand{\coloneqq}{:=}

\newcommand{\sbt}{\,\begin{picture}(-1,1)(-1,-3)\circle*{2}\end{picture}\ }

\newcommand{\ie}{{\it{i.e. }}}
\newcommand{\ii}{{\rm i}}
\newcommand{\orb}{{\rm orb}}

\thanks{The authors are partially supported by proyecto FONDECYT Regular N. 1130572, GDRE-GRIFGA and GDR TLAG}
\begin{document}

\begin{abstract} 
We prove that the Borcea--Voisin mirror pairs of Calabi--Yau threefolds
admit projective birational models that satisfy the Berglund--H\"ubsch--Chiodo--Ruan  
transposition rule. This shows that the two mirror constructions 
provide the same mirror pairs, as soon as both can be defined.
\end{abstract}

\maketitle
 
\section{Introduction}

A famous family of Calabi--Yau threefolds $Z_{E,S}$ has been defined by Borcea~\cite{borcea} and Voisin~\cite{voisin}
from the data of an elliptic curve $E$ and a K3 surface $S$ with a non-symplectic involution $\sigma_S$. 
The varieties $Z_{E,S}$ are defined as the crepant resolutions of the quotients $\frac{E\times S}{\langle \sigma_E\times \sigma_S\rangle}$, where  $\sigma_E=-\id_E$.
The geometry of $Z_{E,S}$ is described by a triple $(r,a,\delta)$ which is obtained from the lattice-theoretical invariants
of the non-symplectic involution on $S$. 
When the surface $S$ varies in the moduli space of K3 surfaces with a non-symplectic involution, 
one obtains moduli spaces $\cZ_{r,a,\delta}$ of Calabi--Yau threefolds indexed by these triples. 
The mirror phenomenon proved by Borcea~\cite{borcea} and Voisin~\cite{voisin}
consists in the observation that the moduli spaces $\cZ_{r,a,\delta}$  and $\cZ_{20-r,a,\delta}$  
form a mirror pair, in the sense that generic elements $Z\in\cZ_{r,a,\delta}$  and $Z'\in\cZ_{20-r,a,\delta}$
satisfy important properties expected for mirror Calabi--Yau varieties, in particular their Hodge diamonds 
are related by a rotation. 
The Borcea--Voisin (BV for short) mirror of $Z_{E,S}$ is 
constructed by taking the lattice mirror of the K3 surface $S$ (see Sections~\ref{ss:lattice_mirror} \& \ref{s:BV}).

In a completely different setup, Chiodo--Ruan~\cite{cr} proved a mirror theorem for Calabi--Yau varieties that are
crepant resolutions of  certain hypersurfaces in weighted projective spaces. The construction is very explicit (see Section~\ref{ss:BHCR}): the mirror of a hypersurface
is obtained by using a generalization of the Berglund--H\"ubsch transposition rule~\cite{bh}  and results of Krawitz~\cite{krawitz}. 
In the sequel, this construction is called BHCR mirror.

The goal
of this paper is to explain how these two constructions of mirror pairs are related to each other. 
The question is a very natural one:
assuming that we have a birational projective model of  $Z_{E,S}$, given as a Calabi--Yau weighted 
hypersurface $W_{E,S}$, is it true that the BHCR mirror of $W_{E,S}$ is birational to the BV mirror of $Z_{E,S}$?

The present paper
is the final step of a project that started with our previous paper~\cite{ABS}, where we proved that lattice mirrors
of K3 surfaces with a non-symplectic involution can be obtained by the BHCR construction (see Theorem~\ref{th:mirrorK3}).
We give a positive answer to the above question in Theorem~\ref{th:main}.
This requires first to produce birational projective models of the varieties $Z_{E,S}$ to which the BHCR construction can be applied
(see Section~\ref{ss:birproj}). This follows an original idea of Borcea but we have to be more precise on the choices
 of the equations: here we make use of the results in \cite{ABS}, where we provided Delsarte type equations for K3 surfaces
with a non-symplectic involution.
Secondly, we have to control the behavior of the transposed groups that occur in the BHCR construction under
these birational equivalences, this is the main technical point (see Section~\ref{ss:behavior}). Finally, we answer
a more general  version of the initial question by considering quotients (see Section~\ref{ss:main}). The Appendix in collaboration
with Sara A. Filippini describes BHCR mirror pairs for elliptic curves.

\section{Preliminary results}

\subsection{Lattice mirror symmetry for K3 surfaces}\label{ss:lattice_mirror}

Non-symplectic involutions on K3 surfaces 
have been classified by Nikulin \cite{nikulinfactor}, we briefly recall the main results.
Let $S$ be a complex projective K3 surface with a non-symplectic involution $\sigma$.
We denote by $\sigma^\ast$ its action on the cohomology lattice $H^2(S,\IZ)$ and
we consider the invariant sublattice
\[
L(\sigma)\coloneqq\{x\in H^2(S,\mathbb Z)\,|\, \sigma^*(x)=x\}.
\]
The lattice $L(\sigma)$ is $2$-elementary, \ie 
its discriminant group $L(\sigma)^\vee/L(\sigma)$ 
is isomorphic to $\left(\IZ/2\IZ\right)^{\oplus a}$ for some non-negative integer $a$
and $L(\sigma)$ is uniquely determined up to isometry by the triple $(r,a,\delta)$,
where $r$ is the rank of $L(\sigma)$ and $\delta\in \{0,1\}$ is 
zero if and only if $x^2\in \IZ$ for any $x\in L(\sigma)^\vee$.
In the sequel we denote this lattice by~$L(r,a,\delta)$.
If $(r,a,\delta)\notin\{(10,8,0), (10,10,0)\}$
then the fixed locus of $\sigma$ in $S$ is a union of smooth disjoint curves
$C_g\cup E_1\cup\cdots \cup E_k$,
where $C_g$ is a curve of genus $g\geq 0$,
the curves $E_i$ are rational and we have (see \cite[Theorem~4.2.2]{nikulinfactor}):
\[
2g=22-r-a,\quad 2k=r-a.
\] 
We now recall the construction of a mirror symmetry for
moduli spaces of K3 surfaces with a non-symplectic involution,
following works of Dolgachev and Nikulin~\cite{dn,nikulinmir,dolgachevmirror}, Voisin~\cite{voisin}
and Borcea~\cite{borcea}.
Let $L_{K3}\coloneqq U^3\oplus E_8^2$ be the K3 lattice.
Consider a  primitive sublattice $M$ of $L_{K3}$ of signature 
$(1,r-1)$ with $1\leq r\leq 19$.
An {\em $M$-polarized K3 surface} is a pair $(S,j)$,
where $S$ is a K3 surface and $j\colon M\hookrightarrow \Pic(S)$ 
is a primitive lattice embedding.
By taking periods one can define a coarse moduli space $\mathcal K_M$ for $M$-polarized K3 surfaces
(see \cite{dolgachevmirror}) which is an irreducible quasi-projective variety 
of dimension $20-r$.
By construction the Picard lattice of a 
generic K3 surface in $\mathcal K_M$ is isomorphic to $M$.
Assume that the orthogonal complement $M^{\perp}$ of $M$ in $L_{K3}$ 
has a decomposition
\begin{equation}\label{u}
M^{\perp}\cong U\oplus M'.
\end{equation}
Identifying $M'$ with its image via this isomorphism, 
we obtain a primitive embedding of $M'$ in the K3 lattice, 
so that we can consider the moduli space $\mathcal K_{M'}$ as before. 
The following conditions hold for $S$ generic in $\mathcal K_M$ and $S'$  generic in $\mathcal K_{M'}$:
\[
\dim (\mathcal K_{M'})=\rk \Pic(S)=r,\quad  \dim (\mathcal K_{M})=\rk \Pic(S')=20-r.
\] 
By exchanging the roles of $M$ and $M'$ we get a duality between moduli spaces 
of lattice polarized K3 surfaces, which will be called \emph{lattice mirror symmetry}, 
where the dimension of the moduli space 
 and the rank of the Picard group are exchanged  
(more evidence is given in \cite{dolgachevmirror}).
In the special case where $M=L(r,a,\delta)$ we have the following result:

\begin{theorem}\cite[Theorem~2.5,\S 2.3]{voisin}\label{th:mirrorK3} If $(r,a,\delta)\neq (14,6,0)$
and $r+a\leq 20$ then condition {\rm (\ref{u})} holds for the lattice $L(r,a,\delta)$
and we have $L(r,a,\delta)'\cong L(20-r,a,\delta)$. Moreover, the generic K3 surface in $\cK_{L(r,a,\delta)}$
admits a non-symplectic involution whose invariant lattice is isometric to $L(r,a,\delta)$.
\end{theorem}

\subsection{The Berglund--H\"ubsch--Chiodo--Ruan  mirror orbifolds}\label{ss:BHCR}

We recall a mirror construction of Chiodo--Ruan~\cite{cr} based on the Berglund--H\"ubsch
transposition rule~\cite{bh}  for hypersurfaces in weighted projective spaces and important results of Krawitz~\cite{krawitz}.

Let $W$ be a polynomial of Delsarte type, 
\ie having $n$ monomials in $n$ variables. 
Up to rescaling its variables we can assume 
that the coefficients of $W$ are all equal to one, 
so that $W$ is characterized by its matrix of exponents $A=(a_{ij})$:
\[
W(x_1,\dots,x_n)=\sum_{i=1}^n \prod_{j=1}^n x_j^{a_{ij}}.
\]
Denoting by ${\bf 1}$ the column vector with all entries equal to one 
we define the \emph{weights sequence} $w\coloneqq (w_1,\dots,w_n)$ as the smallest integer positive multiple 
of the \emph{charges sequence} $q\coloneqq A^{-1}{\bf 1}$.
Putting $w=dq$, $W$ is a $w$-weighted homogeneous polynomial of degree $d$. 
We make the following assumptions:
the matrix $A$ is invertible over $\IQ$,
$w$ is {\em normalized} (\ie
$\gcd(w_1,\dots,\widehat{w_i}, \dots, w_n)=1$ for all $i$),
$W$ is {\em non-degenerate} 
(\ie the affine cone defined by $W$ in $\IC^{n}$ is smooth outside the origin),  
and the {\em Calabi--Yau condition} $\sum\limits_{i=1}^n w_i=d$ is satisfied.
Under these conditions the zero set $X_W$ of $W$ defines
a (singular) Calabi--Yau variety in the weighted projective space
$\IP(w)\coloneqq\IP(w_1,\dots,w_n)$ 
(see for example~\cite[Lemma 1.11]{cortigo} and \cite[\S 2]{ABS}).

Let $\Aut(W)$ be the group of linear automorphisms 
of $\IC^n$ which are diagonal in the coordinates $(x_i)_{i=1,\ldots,n}$ and  which leave $W$ invariant.
Let $\SL(W)$ be the subgroup of $\Aut(W)$ containing the automorphisms of determinant one and
$J_W$ be the subgroup of $\SL(W)$ which acts trivially on $\IP(w)$. 
The group $J_W$ is cyclic,  generated 
by  $j_W\coloneqq\diag\left((\exp(2\ii\pi q_j))_{j=1,\dots,n}\right)$,
where $q=(q_1,\ldots,q_n)$. 
For any subgroup $G\subset \SL(W)$ we put $\widetilde{G}\coloneqq G/J_W$.
Since the action of $\widetilde{G}$ is trivial on the volume form
(see \cite[Proposition~2.3]{ABS}), the quotient $X_W/\widetilde{G}$ is a Calabi--Yau orbifold.
The group $\Aut(W)$ 
can be identified with the group
 $A^{-1}\IZ^n/\IZ^n$ {\it via} the 
map
\[
v=(v_1,\ldots,v_n)\in A^{-1}\IZ^n/\IZ^n\mapsto \diag\left((\exp(2\ii\pi v_j))_{j=1,\ldots,n}\right)\in\Aut(W).
\]
Under this identification $q$ corresponds to $j_W$ 
and $v\in \SL(W)$ if and only if ${\sum\limits_{i=1}^n v_i\in\IZ}$.
We thus identify the group $\widetilde{\SL}(W)$ with the quotient 
\[
\left.\left\{v\in A^{-1}\IZ^n/\IZ^n\,\Big|\, \sum\limits_{i=1}^n v_i\in \IZ\right\}\right/\langle q\rangle. 
\]

Let $W^{\top}$ be the polynomial 
whose matrix of exponents is the transposed matrix $A^{\top}$.
It defines in a similar manner a Calabi--Yau variety $X_{W^{\top}}$ 
in the weighted projective space $\IP(w)^{\top}:=\IP(w_1^{\top},\dots,w_n^{\top})$
where $w^{\top}$ is the smallest positive multiple 
of ${q^{\top}\coloneqq(A^{\top})^{-1}\bf 1}$ with integer entries. 
There is a well-defined bilinear form
\[
\widetilde{\SL}(W^\top)\times \widetilde{\SL}(W)\to \IQ/\IZ,\quad 
(u,v)\mapsto u^\top Av.
\]
Given a subgroup $\widetilde{G}\subset\widetilde{\SL}(W)$,
its {\em transposed group} $\widetilde{G}^\top$ is defined as the orthogonal complement of $\widetilde{G}$
with respect to this bilinear form:
\[
\widetilde{G}^\top:=\{u\in \widetilde{\SL}(W^\top)\,|\, u^\top Av=0\quad \forall v\in \widetilde{G}\}.
\]
The original definition of the transposed group is due to Krawitz~\cite{krawitz} and has been reinterpreted by
many authors, we refer to~\cite{ABS} for a discussion of equivalent definitions.
By \cite[Corollary 3.7]{ABS} we have
\[
|\SL(W)|=\frac{|\det(A)|}{\deg(W^\top)},\quad |\widetilde{\SL}(W)|=\frac{|\det(A)|}{\deg(W)\deg(W^\top)}.
\]

The pair $\left(X_W/\widetilde{G}, X_{W^\top}/\widetilde{G}^\top\right)$ is
called a BHCR mirror pair in virtue of the following result.

\begin{theorem}\cite[Theorem~2]{cr}\label{th:bhcr} The Calabi--Yau orbifolds $[X_W/\widetilde{G}]$ and $[X_{W^\top}/\widetilde{G}^\top]$ satisfy
$$
H_\orb^{p,q}\left([X_W/\widetilde{G}],\IC\right)\cong H_\orb^{n-2-p,q}\left([X_{W^\top}/\widetilde{G}^\top],\IC\right)\quad \forall p,q,
$$
where $H_\orb(-,\IC)$ stands for the orbifold cohomology of Chen and Ruan.
\end{theorem}

The theorem is trivial when $n=3$ (elliptic curves) or $n=4$ (K3 surfaces).
We recall some results concerning the properties of the transposition rule in these two situations.
Assume that the polynomial $W$ takes the form: 
 \[
 W=x_0^2+f(x_1,x_2).
 \]
We denote by $E:=X_W$ the elliptic curve defined by $W$ in $\IP(w_0, w_1, w_2)$, 
it admits an involution induced by $x_0\mapsto -x_0$. By results of the Appendix we have:

 \begin{lemma}\label{ell}
Under the above assumptions, we have $E/\widetilde{\SL}(W)\cong E$.
 \end{lemma}

Assume now that the polynomial $W$ takes the form:
 \[
W=y_0^2+g(y_1,y_2,y_3).
 \]
The minimal resolutions of $X_W/\widetilde{G}$ and $X_{W^\top}/\widetilde{G}^\top$ are
K3 surfaces, that we denote by $S_G$ and $S^\top_{G^\top}$,  which admit a non-symplectic
involution induced by $y_0\mapsto -y_0$. 

\begin{theorem}\cite[Theorem~1.1]{ABS}\label{th:K3} Under the above assumptions,
the K3 surfaces $S_G$ and $S^\top_{G^\top}$ belong to lattice mirror families
$\cK_{L(r,a,\delta)}$ and $\cK_{L(20-r,a,\delta)}$. 
\end{theorem} 

We refer to \cite{ABS} for a complete classification of these K3 surfaces
together with the corresponding values of the triple $(r,a,\delta)$.

\section{Borcea--Voisin mirror threefolds}\label{s:BV}

Let $E$ be an elliptic curve and $S$ be a K3 surface carrying 
non-symplectic automorphisms of order $m$ denoted $\sigma_E$ and $\sigma_S$  respectively.
Denoting by  $\omega_E$ a generator of  $H^{1,0}(E)$ and by $\omega_S$ a generator of $H^{2,0}(S)$ 
we assume that the automorphisms act as follows:
\[
\sigma_E^*(\omega_E)=\zeta^{m-1}_m\omega_E,\quad \sigma_S^*(\omega_S)=\zeta_m\omega_S,
\]
where $\zeta_m$ is a primitive $m$-th root of the unity.
 Clearly $m\in \{2,3,4,6\}$. 
 Consider the quotient
 \[
 Y_{E,S}:=\left(E\times S\right)/\langle\sigma_E\times \sigma_S\rangle. 
 \]
By a result of Bridgeland--King--Reid~\cite[Theorem 1.2]{BKR} the variety $Y_{E,S}$ admits a crepant resolution 
of singularities. Moreover its Hodge numbers do not depend on the chosen resolution (see~\cite{batyrev2,yasuda}).
The following result has been first proved in the case $m=2$ independently
by Borcea~\cite{borcea} and Voisin~\cite{voisin} and has been generalized for $m>2$ by Cattaneo--Garbagnati~\cite{CG}.

\begin{theorem} 
Any crepant resolution of singularities $Z_{E,S}$ of $Y_{E,S}$ is a Calabi--Yau manifold 
whose Hodge numbers depend only on the topological properties of the fixed loci of $\sigma_S^i$, $i=1,\ldots,m-1$.
\end{theorem}

In the sequel we consider the case $m=2$. We denote by $\cZ_{r,a,\delta}$ the family
of Borcea--Voisin Calabi--Yau threefolds of type $Z_{E,S}$ for $S\in\cK_{L(r,a,\delta)}$.
The notation does not take care of the elliptic curve: there is no comment on this choice in Borcea's 
work~\cite{borcea}, whereas Voisin~\cite{voisin} gives more details.
Our main result (see  Section~\ref{ss:main}) makes necessary an explicit choice of the elliptic curve.
 
For any $Z\in\cZ_{r,a,\delta}$ we have (see for instance~\cite{borcea,voisin}):
$$
h^{1,1}(Z)=5+3r-2a,\quad h^{2,1}(Z)=65-3r-2a.
$$
From this computation we see that lattice mirror symmetry for K3 surfaces naturally 
implies a mirror symmetry for Borcea--Voisin Calabi--Yau 
threefolds. If the triple $(r,a,\delta)$ satisfies the conditions of Theorem~\ref{th:mirrorK3},
generic elements ${Z\in\cZ_{r,a,\delta}}$ and ${Z'\in\cZ_{20-r,a,\delta}}$ satisfy:
\[
h^{1,1}(Z)=h^{2,1}(Z'),\quad h^{1,1}(Z')=h^{2,1}(Z).
\]
Deepest motivations supporting the fact that the families $\cZ_{r,a,\delta}$ and $\cZ_{20-r,a,\delta}$ 
are mirror symmetric have been given in \cite{voisin} or \cite[\S 4.4.2]{coxkatz}.
In the next section we provide further evidence for this,
showing that the Borcea--Voisin mirror Calabi--Yau threefolds can be obtained
using the BHCR transposition rule.

\section{BHCR mirrors of Borcea--Voisin threefolds}

\subsection{Birational projective models}\label{ss:birproj}

In order to relate the Borcea--Voisin threefolds $Z_{E,S}$ to the BHCR
setup, we need to provide birational models for the quotient varieties $Y_{E,S}$ as hypersurfaces in weighted
projective spaces given by Delsarte type polynomials. For this we introduce the so-called \emph{twist maps}
already used by Borcea~\cite{borcea} (but the equations obtained there were not of Delsarte type)
and further studied in \cite{GKY,HS}. Consider two hypersurfaces in weighted projective 
spaces defined as follows:
\begin{align*}
X_1&\coloneqq\{x_0^\ell+f(x_1,\dots,x_m)=0\}\subset \IP(u_0,u),\\
X_2&\coloneqq\{y_0^\ell+g(y_1,\dots,y_n)=0\}\subset \IP(v_0,v),
\end{align*}
where $\ell\geq 2$, $u=(u_1,\dots,u_m)$ and $v=(v_1,\dots,v_n)$. We assume that ${\gcd(u_0,v_0)=1}$.
Let $s_0,t_0\in \IZ$ be such that $0\leq s_0<v_0,\ 0\leq t_0<u_0$ and 
\[
s_0u_0+1\equiv 0\mod v_0,\quad t_0v_0+1\equiv 0\mod u_0.
\]
We put $s\coloneqq(s_0u_0+1)/v_0$ and $t\coloneqq(t_0v_0+1)/u_0$.
The {\em twist map} associated to this data is the rational map
\[
\Phi\colon \IP(u_0,u)\times \IP(v_0,v)\dasharrow \IP(v_0u,u_0v)
\]
defined by:
\[
\Phi((x_0,x), (y_0,y))=
\left((x_0^{s_0}y_0^{t})^{u_1}x_1,\dots,(x_0^{s_0}y_0^t)^{u_m}x_m,(x_0^sy_0^{t_0})^{v_1}y_1,\dots, (x_0^sy_0^{t_0})^{v_n}y_n\right).
\]

\begin{remark}\label{rm:nat} From this expression it is clear that $\Phi$ is rational.
By multiplying the coordinates by $x_0^{-\frac{s}{u_0}}y_0^{-\frac{t_0}{u_0}}\in \mathbb C^*$  
the map $\Phi$ takes the more natural form:
\[
\Phi((x_0,x), (y_0,y))=\left(\left(\frac{y_0}{x_0}\right)^{\frac{u_1}{u_0}}x_1,\dots,\left(\frac{y_0}{x_0}\right)^{\frac{u_m}{u_0}}x_m,y_1,\dots,y_n\right).
\]

\end{remark}

We still denote by $(x,y)\coloneqq(x_1,\ldots,x_m,y_1,\ldots,y_n)$ the coordinates in $\IP(v_0u,u_0v)$. 
Consider the restriction map $\Phi_{|X_1\times X_2}\colon X_1\times X_2\dasharrow \IP(v_0u,u_0v)$.
The closure of its image is clearly the irreducible hypersurface
$$
X\coloneqq \{f(x)-g(y)=0\}\subset \IP(v_0u,u_0v).
$$
Observe that, since ${\gcd(u_0,v_0)=1}$, the  weighted projective space $\IP(v_0u,u_0v)$  is normalized.
Denote by $\mu_\ell$ the cyclic group of $\ell$-th roots of the unity. The following result is clear:

\begin{prop}\label{prop:model_twist}
The dominant rational map $\Phi_{|X_1\times X_2}\colon X_1\times X_2\dasharrow X$ has  degree~$\ell$ 
and induces a birational equivalence between $X$ and the quotient $(X_1\times X_2)/\mu_\ell$,
where the cyclic group $\mu_\ell$ acts by
\[
\gamma\cdot ((x_0,x),(y_0,y))=((\gamma x_0,x),(\gamma y_0,y))\quad \forall\gamma\in \mu_{\ell}.
\]
\end{prop}

\begin{lemma}\label{qs}
Assume that $X_1\subset \IP(u_0,u)$ and $X_2\subset \IP(v_0,v)$ 
are non-degenerate Calabi--Yau varieties and that $m+n\geq 5$.
Then $X\subset \IP(v_0u,u_0v)$ is non-degenerate 
and it is a Calabi--Yau variety if and only if $\ell=2$.
\end{lemma}

\begin{proof}
The assertion about the non-degeneracy is clear. 
By \cite[Proposition 6]{dimcaart} $X$ is also well-formed.
Thus by \cite[Lemma 1.11]{cortigo} 
$X$ is a Calabi--Yau variety if and only 
if satisfies the Calabi--Yau condition.
Since $X_1$ and $X_2$ are
Calabi--Yau varieties, then $(\ell-1) u_0=u_1+\dots+u_m$ and
$(\ell-1)v_0=v_1+\dots+v_n$.
Since $X$ has degree $u_0v_0\ell$, the variety $X$ is Calabi--Yau if and only if 
the following condition holds:
\[
u_0v_0\ell=v_0\sum_{i=1}^m u_i+u_0\sum_{i=1}^n v_i=2u_0v_0(\ell -1),
\]
which implies $\ell=2$. This also follows from the fact 
that the action of the group $\mu_\ell$ on $X_1\times X_2$
is non-symplectic except if $\ell=2$.
\end{proof}

In the particular case of Borcea--Voisin threefolds we deduce the following
birational models, already given by Borcea \cite{borcea}. 

\begin{prop}\label{prop:model_borcea}
Let $E\coloneqq\{x_0^2+f(x_1,x_2)=0\}\subset \IP(u_0,u_1,u_2)$ 
be an elliptic curve and 
$S\coloneqq\{y_0^2+g(y_1,y_2,y_3)=0\}\subset \IP(v_0,v_1,v_2,v_3)$ 
be a K3 surface.
\begin{enumerate}
\item If $(u_0,u_1,u_2)=(2,1,1)$ and $v_0$ is odd, then 
$Y_{E,S}$ is birational to the hypersurface
\[
\{f(x)-g(y)=0\}\subset \IP(v_0,v_0,2v_1,2v_2,2v_3).
\]
\item If $(u_0,u_1,u_2)=(3,2,1)$ and $v_0$ is not divisible by $3$,
then $Y_{E,S}$ is birational to the hypersurface
\[
\{f(x)-g(y)=0\}\subset \IP(2v_0,v_0,3v_1,3v_2,3v_3).
\]
\end{enumerate}
\end{prop}
 
\begin{remark}Note that if $v_0$ is divisible by $6$ this construction does not provide
birational projective models, in fact the twist map is not well defined.
\end{remark}

\subsection{Behavior of the transposed groups}\label{ss:behavior}

Let us recall the notation. We consider an elliptic curve $E$ given by an equation of type
$$
W_E\coloneqq x_0^2+f(x_1,x_2)\subset\IP(u_0,u_1,u_2)
$$
and a K3 surface $S$ given by an equation of type
$$
W_S\coloneqq y_0^2+g(y_1,y_2,y_3)\subset\IP(v_0,v_1,v_2,v_3).
$$
Assuming that $\gcd(u_0,v_0)=1$, we consider the birational model for the Borcea--Voisin
Calabi-Yau threefold $Y_{E,S}\subset \IP(v_0u, u_0v)$ given by the equation
\[
W_{E,S}\coloneqq f(x_1,x_2)-g(y_1,y_2,y_3).
\]
We further assume that the polynomials $W_E$ and $W_S$ are of Delsarte type, non-degenerate, 
defined by invertible matrices,
and that they define $E$ and $S$ in normalized weighted projective spaces. It is easy to check that under these
assumptions the polynomial~$W_{E,S}$ satisfies the same properties.

We denote by $A_E$ the matrix of exponents of $W_E$ 
and by $\hat A_E$ its submatrix obtained deleting the first row and column.
We define similarly $A_S$  and $\hat A_S$.
Clearly the matrix  of exponents of $W_{E,S}$ is $A_{E,S}=\diag\left(\hat A_E,\hat A_S\right)$.

We denote by $E^\top$ the elliptic curve defined by the transposed polynomial~$W_E^\top$ in the weighted projective space $\IP(u_0,u)^\top=\IP(u_0^\top,u^\top)$. 
We denote similarly by~$S^\top$
the (singular) K3 surface defined by~$W_S^\top$ in~$\IP(v_0^\top,v^\top)$. 
\begin{lemma}\label{lemma2}
 If $\gcd(u_0^\top,v_0^\top)=1$, then
$W_{E^\top,S^\top}=W_{E,S}^\top$ and this equation defines a hypersurface in the weigthed
 projective space $\IP(v_0u,u_0v)^\top=\IP(v_0^\top u^\top,u_0^\top v^\top)$.
\end{lemma}

\begin{proof}
Since $A_{E,S}^\top=\diag(\hat A_E^\top, \hat A_S^\top)$ we have $W_{E^\top,S^\top}=W_{E,S}^\top$.
The charges of $W_{E^\top,S^\top}$ are by definition 
\[
q^\top\coloneqq (A_{E,S}^\top)^{-1}{\bf 1}=\left(
\begin{matrix}
(\hat A_E^\top)^{-1}{\bf 1}\\
(\hat A_S^\top)^{-1}{\bf 1}\\
\end{matrix}
\right).
\]
The degree $d_{E^\top}=2u_0^\top$ of $E^\top$ satisfies $d_{E^\top}(\hat A_E^\top)^{-1}{\bf 1}=u^\top$
and the degree $d_{S^\top}=2v_0^\top$ of $S^\top$ satisfies $d_{S^\top}(\hat A_S^\top)^{-1}{\bf 1}=v^\top$,
hence
$$
q^\top=\frac{1}{2}\left(\begin{matrix}
(u_0^\top)^{-1}u^\top\\
(v_0^\top)^{-1}v^\top
\end{matrix}
\right).
$$
Since $\gcd(u_0^\top,v_0^\top)=1$, the smallest integer multiple of $q^\top$ is 
\[
(2u_0^\top v_0^\top)q^\top=
\left(
\begin{matrix}
 v_0^\top u^\top\\
u_0^\top v^\top
\end{matrix}
\right),
\]
hence $W_{E^\top,S^\top}$ has degree $2u_0^\top v_0^\top$ and the transposed weights
satisfy 
$
(v_0u,u_0v)^\top=(v_0^\top u^\top,u_0^\top v^\top).
$
\end{proof}

\begin{lemma}\label{lemma3} Under the above assumptions, there is an isomorphism
 \[
\tilde\theta\colon \widetilde{\SL}(W_{E,S})\lra \widetilde{\SL}(W_E)\times\widetilde{\SL}(W_S).
\]
\end{lemma}
\begin{proof} As above we identify an element of $\Aut(-)$ with the vector of its diagonal entries.
Let $\SL^{\pm 1}(W_{E,S})$ be the subgroup of $\SL(W_{E,S})$ consisting of elements $(\alpha,\beta)\coloneqq(\alpha_1,\alpha_2,\beta_1,\beta_2,\beta_3)$
such that $\alpha_1\alpha_2=\beta_1\beta_2\beta_3=\pm 1$. Let $\alpha_0=\pm 1$ be such that 
$\alpha_0\alpha_1\alpha_2=1$. Clearly $\tilde\alpha\coloneqq (\alpha_0,\alpha)\in\SL(W_E)$. 
Denoting similarly $\tilde\beta\in\SL(W_S)$ we get an exact sequence
\[
\xymatrix{
1\ar[r] & \SL^{\pm 1}(W_{E,S})\ar[r]^{\iota\qquad} & 
\SL(W_E)\times\SL(W_S)\ar[r]^{\qquad \pi} & \{\pm 1\},
}
\]
where $\iota(\alpha,\beta)=(\tilde \alpha, \tilde \beta)$ 
and $\pi(\gamma,\delta)=\gamma_0\delta_0$, with $\gamma=(\gamma_i)_{i=0,1,2}\in\SL(W_E)$
and $\delta=(\delta_i)_{i=0,1,2,3}\in\SL(W_S)$.
Since $E$ has degree $2u_0$, its first charge is $\frac{1}{2}$, 
so the element $j_E$ has first coordinate $-1$,
and similarly for $S$. This implies that the map~$\pi$ is surjective.
Using the results recalled in Section~\ref{ss:BHCR} and Lemma~\ref{lemma2} we deduce:  
\begin{align*}
|\SL^{\pm 1}(W_{E,S})|&=
\frac{| \SL(W_E)|\cdot | \SL(W_S)|}{2}=
\frac{|\det(\hat A_E)\det(\hat A_S)|}{2u_0^\top v_0^\top}\\
&=
\frac{|\det(A)|}{\deg(W_{E,S}^\top)}=|\SL(W_{E,S})|.
\end{align*}
It follows that $\SL^{\pm 1}(W_{E,S})=\SL(W_{E,S})$. Since $\gcd(u_0,v_0)=1$, either $u_0$ or $v_0$ is odd.
The map $s\colon\{\pm 1\}\to \SL(W_E)\times\SL(W_S)$ defined by $s(-1)=(j_E^{u_0},\id)$ if $u_0$ is odd,
or by $s(-1)=(\id,j_S^{v_0})$ if $v_0$ is odd, is a section
of $\pi$. It induces an isomorphism
 \[
\psi\colon\SL(W_{E,S})\times \{\pm 1\}\to  \SL(W_E)\times\SL(W_S),
\quad ((\alpha,\beta),- 1)\mapsto \iota(\alpha,\beta)\cdot s(-1).
 \]
The generator $j=(j_1,j_2)$ of $J_{W_{E,S}}$ is such that $\tilde\jmath_1=j_E$ and $\tilde \jmath_2=j_S$, hence
 $\psi\left(J_{W_{E,S}}\times \{\pm 1\}\right)=J_{W_E}\times J_{W_S}$ (the inclusion is clear and the groups have
the same order). We thus get an isomorphism
\[
\tilde\theta\colon\widetilde{\SL}(W_{E,S})\to \widetilde{\SL}(W_E)\times\widetilde{\SL}(W_S).
\]
\end{proof}

\begin{remark}\label{rm:det} Since $j_E$ and $j_S$ have determinant $1$ and first entry equal to $-1$,
both $j_1$ and $j_2$ have determinant  equal to $-1$. By multiplying by~$j$, we
can thus assume that any class $[(\alpha,\beta)]\in\widetilde{\SL}(W_{E,S})$ has a representative $(\alpha,\beta)$
such that $\alpha_1\alpha_2=\beta_1\beta_2\beta_3=1$. With these choices the map $\tilde\theta$ is given by:
\[
\tilde\theta [(\alpha,\beta)]= \left([(1,\alpha)],[(1,\beta)]\right).
\]
\end{remark}

Assuming that $\gcd(u_0^\top,v_0^\top)=1$, we denote by 
$$\tilde\theta^\top\colon\widetilde{\SL}(W_{E,S}^\top)\to \widetilde{\SL}(W_E^\top)\times\widetilde{\SL}(W_S^\top)
$$
the map defined similarly as $\tilde\theta$ at the level of the transposed potentials.
For any subgroups $G_E\subset\widetilde{\SL}(W_{E})$ and $G_S\subset\widetilde{\SL}(W_{S})$, we define
$$
G_{E,S}\coloneqq \tilde\theta^{-1}(G_E\times G_S).
$$

\begin{lemma}\label{lemma4} Under the above assumptions we have $\tilde\theta^\top(G_{E,S}^\top)=G_E^\top\times G_S^\top$.
\end{lemma}

\begin{proof}
Using the results of Section~\ref{ss:BHCR}, we identify $\widetilde{\SL}(W_{E,S})$ 
with the subgroup of the quotient of $A_{E,S}^{-1}\IZ^5/\IZ^5$ by $\langle q\rangle$ 
consisting of those vectors $u$ such that $\sum_i u_i \in \IZ$.  
We do a similar identification for the groups $\widetilde \SL(W_E)$, $\widetilde \SL(W_S)$ and the corresponding groups  
for the transposed polynomials.
With these identifications we have
\[
 G_{E,S}^\top=\{u\in \widetilde{\SL}(W_{E,S}^\top)\,|\, u^\top A_{E,S}v=0\quad \forall v\in  G_{E,S}\}.
\]
By Remark \ref{rm:det}  any element  $v\in  G_{E,S}$ can be represented
as the diagonal join of two diagonal matrices of determinant one, 
explicitly  $v=(\gamma,\delta)$ and:
\[
\tilde\theta(v)=((0,\gamma_1,\gamma_2),(0,\delta_1,\delta_2,\delta_3)), 
\]
where $\tilde \gamma=(0,\gamma_1,\gamma_2)\in   G_E$ and $\tilde\delta=(0,\delta_1,\delta_2,\delta_3)\in  G_S$ 
since $G_{E,S}=\tilde\theta^{-1}( G_E\times  G_S)$. 
Thus $u=(\alpha,\beta)\in G_{E,S}^\top$ if and only if 
$\alpha^\top\hat A_E\gamma=0$ and  $\beta^\top\hat A_S\delta=0$,
\ie $\tilde \alpha\in G_E^\top$ and $\tilde \beta\in G_S^\top$.
This proves the statement.
\end{proof}

\subsection{Main result}\label{ss:main}

The main result of this paper is a construction of Borcea--Voisin mirror pairs by means of the BHCR transposition rule. For this, we start from
an elliptic curve~$E$ and a K3 surface~$S$, under the above assumptions, and we consider subgroups $G_E\subset\widetilde{\SL}(W_{E})$ and $G_S\subset\widetilde{\SL}(W_{S})$. Denoting $G_{E,S}\coloneqq\tilde\theta^{-1}(G_E\times  G_S)$, we know
from Lemma~\ref{ell} and Theorem~\ref{th:K3} that the Borcea--Voisin  Calabi--Yau threefolds
$$
Z_{E/G_E,S_{G_S}} \quad\text{and}\quad Z_{E^\top/G_E^\top,S^\top_{G^\top_S}}
$$
form a mirror pair. We also know from Theorem~\ref{th:bhcr} that the BHCR Calabi--Yau orbifolds
$$
\left[X_{W_{E,S}}/ G_{E,S}\right] \quad\text{and} \quad \left[X_{W^\top_{E,S}}/G_{E,S}^\top\right]
$$
form a mirror pair.

\begin{theorem}\label{th:main} Under the above assumptions, $Z_{E/G_E,S_{G_S}}$ is birational to $X_{W_{E,S}}/G_{E,S}$ and $Z_{E^\top/G_E^\top,S^\top_{G^\top_S}}$ is birational to $X_{W^\top_{E,S}}/G_{E,S}^\top$.
\end{theorem}

\begin{proof}
By assumption $E$ and $S$ are defined by Delsarte type, invertible and non-degenerate polynomials  
\[
 E\coloneqq\{x_0^2+f(x_1,x_2)=0\}\subset \IP(u_0,u),
\quad S\coloneqq\{y_0^2+g(y_1,y_2,y_3)=0\}\subset \IP(v_0,v)
\]
inside normalized weighted projective spaces and
we assume that $\gcd(u_0,v_0)=1$ and $\gcd(u_0^\top,v_0^\top)=1$.
By Propositions~\ref{prop:model_twist} \& \ref{prop:model_borcea} the twist map induces a birational map
$$
\xymatrix{\Phi\colon Y_{E,S}=\frac{E\times S}{\langle \sigma_E\times \sigma_S\rangle}\ar@{-->}^-{\sim}[r]&X_{W_{E,S}}}.
$$
Using Remarks~\ref{rm:nat} \& \ref{rm:det} we see that $G_E\times G_S$ commutes with $\sigma_E\times \sigma_S$ and that~$\Phi$ is equivariant for the action of $G_E\times G_S$ on the source and of $G_{E,S}$ on the target, hence it induces
a birational equivalence
$$
\bar\Phi\colon (E/G_E\times S/G_S)/\langle \sigma_E\times \sigma_S\rangle\to X_{W_{E,S}}/G_{E,S}.
$$
Recall  that $S_{G_S}$ denotes the minimal resolution of $S/G_S$ (see end of Section~\ref{ss:BHCR}), which
still admits a non-symplectic involution, hence we have birational morphisms induced by resolutions of singularities:
$$
Z_{E/G_E,S_{G_S}}\rightarrow Y_{E/G_E,S_{G_S}}\rightarrow (E/G_E\times S/G_S)/\langle\sigma_E\times \sigma_S\rangle.
$$
This proves the first statement. Similarly, using Lemmas~\ref{lemma2} \& \ref{lemma4} we get that the twist map induces a birational map
$$
(E^\top/G_E^\top\times S^\top/G_S^\top)/\langle\sigma_{E^\top}\times \sigma_{S^\top}\rangle \to X_{W_{E^\top,S^\top}}/G_{E,S}^\top= X_{W^\top_{E,S}}/G_{E,S}^\top 
$$
and we deduce the second statement.
\end{proof}

Since birational Calabi--Yau threefolds have the same Hodge numbers (see~\cite{batyrev2}), 
the theorem implies that the two mirror constructions studied
in this paper (the construction of Borcea--Voisin and the BHCR mirror construction) 
provide the same mirror pairs as soon as both are defined. 
The situation is represented in  the following diagram, 
where for simplicity we consider only the case
$G_E=G_S=\{\id \}$. 
We denote by $Z$  the crepant resolution of $E\times S/\langle\sigma_E\times \sigma_S\rangle$. 
Similarly 
we denote by  $Z'$ the crepant resolution of $(E^\top/\widetilde\SL(W_E^\top)\times S^\top/\widetilde\SL(W_S^\top))/\langle \sigma_{E^\top}\times\sigma_{S^\top}\rangle$.

$$
\xymatrixcolsep{9pc}
\xymatrix{
Z\ar[d]_{crepant}\ar@{<~>}[r]^{BV~mirror}&Z'\ar[d]^{crepant}\\
\frac{E\times S}{\langle\sigma_E\times \sigma_S\rangle}\ar@{.>}[d]_{birational}&\frac{E^\top/\widetilde\SL(W_E^\top)\times S^\top/\widetilde\SL(W_S^\top)}{\langle \sigma_{E^\top}\times\sigma_{S^\top}\rangle}\ar@{.>}[d]^{birational}\\
X_{W_{E,S}}\ar@{<~>}[r]^{BHCR~mirror}&X_{W_{E^\top,S^\top}}/\widetilde\SL(W_{E,S}^\top)
}
$$

\vspace{0.2cm}

\begin{example} Consider the elliptic curve and the K3 surface defined by
\[
E\coloneqq\{x_0^2+x_1^4+x_2^4=0\}\subset \IP(2,1,1),
\quad S\coloneqq\{y_0^2+y_1^5y_2+y_2^5y_3+y_3^6=0\}\subset \IP(3,1,1,1).
\]
Observe that $S$ is the double cover of $\IP^2$ 
branched along a smooth sextic curve and that the involution 
$\sigma_S\colon y_0\mapsto -y_0$ 
is the covering involution. 
Since the fixed locus of $\sigma_S$ is a smooth curve of genus $10$, 
the invariants of its fixed lattice are (see Section~\ref{ss:lattice_mirror}) $(r,a,\delta)=(1,1,1)$,
\ie $S\in \mathcal \cK_{L(1,1,1)}$. It follows that $Z_{E,S}\in\cZ_{1,1,1}$ and that 
its Hodge numbers are $h^{1,1}=6$,
 $h^{2,1}=60$ (see Section~\ref{s:BV}).
By Proposition \ref{prop:model_borcea} a birational projective model for the Borcea--Voisin 
Calabi--Yau threefold $Z_{E,S}$ is 
\[
X_{W_{E,S}}=\{x_1^4+x_2^4-y_1^5y_2-y_2^5y_3-y_3^6=0\}\subset \IP(3,3,2,2,2).
\]
Clearly $E=E^\top$ and
\[
S^\top=\{y_0^2+y_1^5+y_1y_2^5+y_2y_3^6=0\}\subset \IP(25,10,8,7).
\]
The group $\widetilde{\SL}(W_{E^\top})$ has order $2$ and it is generated 
by the involution:
$$
\iota\colon(x_0,x_1,x_2)\mapsto (-x_0,x_1,-x_2),
$$ 
while an easy computation shows that $\widetilde{\SL}(W_{S^\top})$ is trivial.
Thus the BHCR mirror of $E$ is $E/\langle\iota\rangle$ and that of $S$ is $S^\top$.
By Theorem~\ref{th:K3} we find that the invariants of the fixed lattice 
of the involution $\sigma_{S^\top}$ 
are $(r,a,\delta)=(19,1,1)$, \ie $S^{\top}\in \cK_{L(19,1,1)}$.
Thus $Z_{E/\langle\iota\rangle,S^\top}\in\cZ_{19,1,1}$ and its Hodge numbers are $h^{1,1}=60$, $h^{2,1}=6$.
The transposed potential of $W_{E,S}$ is
\[
W_{E,S}^{\top}=\{x_1^4+x_2^4-y_1^5-y_1y_2^5-y_2y_3^6=0\}\subset \IP(25,25,20,16,14).
\]
On the other hand, the BHCR mirror of $X_{W_{E,S}}$ is the quotient of 
$X_{W_{E,S}^{\top}}$ by the action of the group $\widetilde{\SL}(W_{E,S}^{\top})$,
which is generated by the involution 
$$
\nu\colon (x_1,x_2,y_1,y_2,y_3)\mapsto (\ii x_1,-\ii x_2,y_1,y_2,y_3).
$$
Through the isomorphism $\tilde\theta$ of Lemma~\ref{lemma3}, the element $\nu$ corresponds to  $(\iota,\id)$
since $(-x_0,x_1,-x_2)=(x_0,\ii x_1,-\ii x_2)$ in $\IP(2,1,1)$.
As in the proof of Theorem~\ref{th:main} 
we get that $Z_{E/\langle\iota\rangle,S^\top}$ is birational to
$X_{W_{E,S}^{\top}}/\langle\nu\rangle$. 
\end{example}

\begin{remark}
Under the above assumptions, taking a subgroup $G\subset\widetilde{\SL}(W_{E,S})$ such that $\tilde\theta(G)$ is
not a direct product, one gets more examples of Calabi--Yau mirror manifolds. 
The twist map $\Phi\colon Y_{E,S}\dasharrow X_{W_{E,S}}$ is
equivariant for the action of $\tilde\theta(G)$ on the source and of $G$ on the target, hence $Y_{E,S}/\tilde\theta(G)$
is birational to $X_{W_{E,S}}/G$. By Lemma~~\ref{lemma4} the group $\tilde\theta^\top(G^\top)$ is not a direct product. By the theorem of Bridgeland--King--Reid~\cite[Theorem~1.2]{BKR}, the Nakamura Hilbert scheme of $\tilde\theta(G)\times \langle\sigma_E\times \sigma_S\rangle$-regular orbits of $E\times S$ is a crepant resolution of $Y_{E,S}/\tilde\theta(G)$, 
hence it is a smooth birational model of the Calabi--Yau orbifold $[X_{W_{E,S}}/G]$. 
Applying the BHCR mirror symmetry we deduce that the Nakamura 
Hilbert scheme of $\tilde\theta^\top(G^\top)\times \langle\sigma_E^\top\times\sigma_S^\top\rangle$-regular orbits of $E^\top\times S^\top$ is a smooth birational model of the mirror Calabi--Yau orbifold $[X_{W_{E,S}^\top}/G^\top]$.
\end{remark}

\section{Borcea--Voisin families with BHCR models}

Looking at the classification of Nikulin~\cite{nikulinfactor}, 
one finds $63$ families of K3 surfaces with a non-symplectic 
involution that admit a lattice mirror. 
For  any such family, the Borcea--Voisin mirror construction is defined. 
Each of these families is determined by a triple $(r,a,\delta)$ and is represented 
in Figure \ref{tree} by either a dot or a star which does not lie 
on the right hand side of the triangle and is distinct from $(r,a,\delta)=(14,6,0)$.
By results of \cite{ABS} exactly $29$ of these families contain Delsarte type models.
In \cite{ABS} it has been proved that for them the  
lattice mirror symmetry for K3 surfaces
is equivalent to the BHCR mirror symmetry.

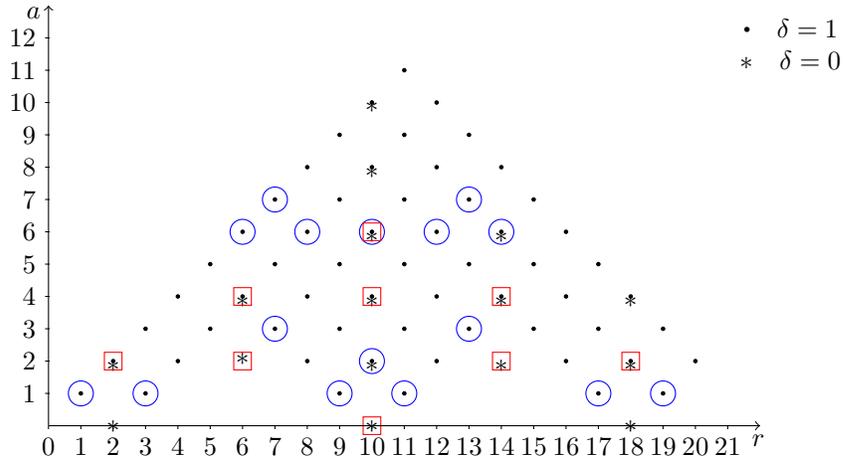
\begin{figure}[!ht]
\hfill $\begin{array}{l} \sbt \quad \delta=1\\ 
\ast \quad \delta=0\end{array}$
\vspace{-1.1cm}

\begin{tikzpicture}[scale=.43]
\filldraw [black] 
(1,1) circle (1.5pt)  node[draw,circle,blue](10){}
(2,0) node[below=-0.20cm]{*} 
(2,2) circle (1.5pt) node[below=-0.15cm]{*} node[draw,red](10){}
 (3,1) circle (1.5pt) node[draw,circle,blue](10){}
 (3,3) circle (1.5pt)
 (4,2) circle (1.5pt)
(4,4) circle (1.5pt)
(5,3) circle (1.5pt)
(5,5) circle (1.5pt)
(6,4) circle (1.5pt)node[below=-0.15cm]{*} node[draw,red](10){}
(6,2)    node[below=-0.23cm]{*} node[draw,red](10){}
(6,6) circle (1.5pt) node[draw,circle,blue](10){}
(7,3) circle (1.5pt) node[draw,circle,blue](10){}
(7,5) circle (1.5pt)
(7,7) circle (1.5pt) node[draw,circle,blue](10){}
(8,2) circle (1.5pt)
(8,4) circle (1.5pt)
(8,6) circle (1.5pt) node[draw,circle,blue](10){}
(8,8) circle (1.5pt)
(9,1) circle (1.5pt) node[draw,circle,blue](10){}
(9,3) circle (1.5pt)
(9,5) circle (1.5pt)
(9,7) circle (1.5pt)
(9,9) circle (1.5pt)
(10,0)  node[below=-0.20cm]{*} node[draw,red](10){}
(10,2) circle (1.5pt)node[below=-0.15cm]{*} node[draw,circle,blue](10){}
(10,4) circle (1.5pt)node[below=-0.15cm]{*} node[draw,red](10){}
(10,6) circle (1.5pt)node[below=-0.15cm]{*} node[draw,circle,blue](10){} node[draw,red](10){}
(10,8) circle (1.5pt)node[below=-0.15cm]{*}
(10,10) circle (1.5pt)node[below=-0.15cm]{*}
(11,1) circle (1.5pt) node[draw,circle,blue](10){}
(11,3) circle (1.5pt)
(11,5) circle (1.5pt)
(11,7) circle (1.5pt)
(11,9) circle (1.5pt)
(11,11) circle (1.5pt) 
(12,2) circle (1.5pt)
(12,4) circle (1.5pt)
(12,6) circle (1.5pt) node[draw,circle,blue](10){}
(12,8) circle (1.5pt)
(12,10) circle (1.5pt) 
(13,3) circle (1.5pt) node[draw,circle,blue](10){}
(13,5) circle (1.5pt)
(13,7) circle (1.5pt) node[draw,circle,blue](10){}
(13,9) circle (1.5pt) 
(14,2)  node[below=-0.15cm]{*} node[draw,red](10){}
(14,4) circle (1.5pt)node[below=-0.15cm]{*} node[draw,red](10){}
(14,6) circle (1.5pt)node[below=-0.15cm]{*} node[draw,circle,blue](10){}
(14,8) circle (1.5pt) 
(15,3) circle (1.5pt)
(15,5) circle (1.5pt)
(15,7) circle (1.5pt)
(16,2) circle (1.5pt)
(16,4) circle (1.5pt)
(16,6) circle (1.5pt) 
(17,1) circle (1.5pt) node[draw,circle,blue](10){}
(17,3) circle (1.5pt)
(17,5) circle (1.5pt)
(18,0)  node[below=-0.2cm]{*}
(18,2) circle (1.5pt)node[below=-0.15cm]{*} node[draw,red](10){}
(18,4) circle (1.5pt) node[below=-0.15cm]{*}
(19,1) circle (1.5pt) node[draw,circle,blue](10){}
(19,3) circle (1.5pt) 
(20,2) circle (1.5pt) 
 ; 
\draw[->] (0,0) -- coordinate (x axis mid) (22,0);
    \draw[->] (0,0) -- coordinate (y axis mid)(0,13);
    \foreach \x in {0,1,2,3,4,5,6,7,8,9,10,11,12,13,14,15,16,17,18,19,20,21}
        \draw [xshift=0cm](\x cm,0pt) -- (\x cm,-3pt)
         node[anchor=north] {$\x$};
          \foreach \y in {1,2,3,4,5,6,7,8,9,10,11,12}
        \draw (1pt,\y cm) -- (-3pt,\y cm) node[anchor=east] {$\y$};
    \node[below=0.2cm, right=4.5cm] at (x axis mid) {$r$};
   \node[left=0.2cm, below=-2.9cm] at (y axis mid) {$a$};
  \end{tikzpicture} 
  \caption{Triples $(r,a,\delta)$ where Borcea--Voisin  and BHCR mirrors are equivalent} \label{tree}
\end{figure}
 
 Looking at the classification tables \cite[\S 7]{ABS} we find that
Theorem \ref{th:main} can be applied to $25$ of these families. 
In Figure \ref{tree} the $25$ triples $(r,a,\delta)$ are marked with 
either a circle (when $\delta=1$) or a square (when $\delta=0$). 
The four bad triples are $(2,0,0)$, $(18,0,0)$, $(4,4,1)$, $(16,4,1)$.
In these cases the Delsarte models of the K3 surface $S$ all lie 
in a weighted projective space $\IP(v_0,v_1,v_2,v_3)$ 
where $6$ divides $v_0$, thus Proposition~\ref{prop:model_borcea} 
does not apply.\\

\begin{center}
{\bf Appendix: BHCR mirror pairs  for elliptic curves}\\
\vspace{3mm}
{\rm{\small{\sc in collaboration with Sara A. Filippini}}}
\end{center}

\vskip 10pt

Using similar techniques as in the paper \cite{ABS} one can easily find a table of equations~$W$  of elliptic curves $E$ in weighted three dimensional projective spaces $\IP(w_0,w_1,w_2)$, that are defined by non-degenerate, invertible potentials. The possible weights are $(1,1,1)$, $(2,1,1)$ and $(3,2,1)$, so $|J_W|=3, 4$ and $6$ respectively. We give below the table with the elliptic curves and their BHCR mirrors (the number of the mirror is given in the parenthesis).
The equations marked with a $\clubsuit$ are those that we use in the paper. 

\begin{table}[ht!]
$$
\begin{array}{r|l|c|c|c}
\mbox{No.} &(w_0,w_1,w_2)&W&|\SL(W)|&|\SL(W)/J_W|\\
\hline
(1)1& (1,1,1)&x_0^3+x_1^3+x_2^3&9&3\\
(7)2&(1,1,1)&x_0^2x_1+x_1^2x_2+x_2^3&3&1\\
(3)3&(1,1,1)&x_0^2x_1+x_1^2x_2+x_2^2x_0&3&1\\ 
(13)4&(1,1,1)&x_0^3+x_1^2x_2+x_2^3&3&1\\
(5)5&(1,1,1)&x_0^3+x_1^2x_2+x_2^2x_1&3&1\\
\hline
(6)6& (2,1,1)~\clubsuit  & x_0^2+x_1^4+x_2^4&8&2\\
(2)7&(2,1,1)&x_0^2+x_0x_1^2+x_1x_2^3&4&1\\
(12)8&(2,1,1) ~\clubsuit & x_0^2+x_1^3x_2+x_2^4&4&1\\ 
(9)9&(2,1,1)&x_0^2+x_0x_1^2+x_2^4&4&1\\
(10)10&(2,1,1)~\clubsuit &x_0^2+x_1^3x_2+x_2^3x_1&4&1\\
\hline
(11)11&(3,2,1)~\clubsuit &x_0^2+x_1^3+x_2^6&6&1\\ 
(8)12&(3,2,1)~\clubsuit  &x_0^2+x_1^3+x_1x_2^4&6&1\\
(4)13&(3,2,1)&x_0^2+x_0x_2^3+x_1^3&6&1\\
\end{array}
$$ 
\vspace{0.2cm}

\caption{The BHCR mirror pairs of elliptic curves}
\end{table}
In case 1 the group $\SL(W)/J_W$ is generated by the automorphism $$\phi\colon (x_0, x_1, x_2)\mapsto (\zeta^2 x_0, \zeta x_1, \zeta x_2)$$ with $\zeta$ a primitive third root of the unity. 
It is easy to see that the quotient $E/\langle \phi \rangle$ is an elliptic curve with equation
of type 5 in the table.  Similarly, in case 6, the group $\SL(W)/J_W$ is generated 
by the automorphism $$\iota\colon (x_0, x_1, x_2)\mapsto (-x_0, x_1, -x_2).$$ 
The quotient $E/\langle \iota \rangle$ is an elliptic curve with equation
of type 2 in the table. 
Observe that in this last case $E$ has a complex multiplication of order four given by
 \[
 \mu: (x_0,x_1,x_2)\mapsto (x_0, \ii x_1,x_2).
 \]
Since $\mu$ commutes with $\iota$, it induces complex multiplication of order four on 
the quotient $E/\langle \iota \rangle$, that is thus isomorphic to $E$.

\vskip 30pt

\noindent
{\small {\sc Sara A. Filippini}, {\sc Institut f\"ur Mathematik
Universit\"at Z\"urich,
Winterthurerstrasse 190
8057 Z\"urich, Switzerland},  {\it E-mail address}: {\tt saraangela.filippini@math.uzh.ch}}
\vskip 30pt

\bibliographystyle{amsplain}
\bibliography{Biblio}

\providecommand{\bysame}{\leavevmode\hbox to3em{\hrulefill}\thinspace}
\providecommand{\MR}{\relax\ifhmode\unskip\space\fi MR }
\providecommand{\MRhref}[2]{%
  \href{http://www.ams.org/mathscinet-getitem?mr=#1}{#2}
}
\providecommand{\href}[2]{#2}
\begin{thebibliography}{10}

\bibitem{ABS}
M.~Artebani, S.~Boissi\`ere, and A.~Sarti, \emph{The
  {B}erglund-{H}\"ubsch-{C}hiodo-{R}uan mirror symmetry for {K}3 surfaces}, J.
  Math. Pures Appl. (9) \textbf{102} (2014), no.~4, 758--781.

\bibitem{batyrev2}
V.~V. Batyrev, \emph{Birational {C}alabi-{Y}au {$n$}-folds have equal {B}etti
  numbers}, New trends in algebraic geometry ({W}arwick, 1996), London Math.
  Soc. Lecture Note Ser., vol. 264, Cambridge Univ. Press, Cambridge, 1999,
  pp.~1--11.

\bibitem{bh}
P.~Berglund and T.~H{\"u}bsch, \emph{A generalized construction of mirror
  manifolds}, Nuclear Phys. B \textbf{393} (1993), no.~1-2, 377--391.

\bibitem{borcea}
C.~Borcea, \emph{{$K3$} surfaces with involution and mirror pairs of
  {C}alabi--{Y}au manifolds}, Mirror symmetry, {II}, AMS/IP Stud. Adv. Math.,
  vol.~1, Amer. Math. Soc., Providence, RI, 1997, pp.~717--743.

\bibitem{BKR}
T.~Bridgeland, A.~King, and M.~Reid, \emph{The {M}c{K}ay correspondence as an
  equivalence of derived categories}, J. Amer. Math. Soc. \textbf{14} (2001),
  no.~3, 535--554.

\bibitem{CG}
A.~Cattaneo and A.~Garbagnati, \emph{{C}alabi--{Y}au 3-folds of
  {B}orcea-{V}oisin type and elliptic fibrations}, \texttt{arXiv:1312.3481}.

\bibitem{cr}
A.~Chiodo and Y.~Ruan, \emph{L{G}/{CY} correspondence: the state space
  isomorphism}, Adv. Math. \textbf{227} (2011), no.~6, 2157--2188.

\bibitem{cortigo}
A.~Corti and V.~Golyshev, \emph{Hypergeometric equations and weighted
  projective spaces}, Sci. China Math. \textbf{54} (2011), no.~8, 1577--1590.

\bibitem{coxkatz}
D.~A. Cox and S.~Katz, \emph{Mirror symmetry and algebraic geometry},
  Mathematical Surveys and Monographs, vol.~68, American Mathematical Society,
  Providence, RI, 1999.

\bibitem{dimcaart}
A.~Dimca, \emph{Singularities and coverings of weighted complete
  intersections}, J. Reine Angew. Math. \textbf{366} (1986), 184--193.

\bibitem{dolgachevmirror}
I.~V. Dolgachev, \emph{Mirror symmetry for lattice polarized {$K3$} surfaces},
  J. Math. Sci. \textbf{81} (1996), no.~3, 2599--2630, Algebraic geometry, 4.

\bibitem{dn}
I.~V. Dolgachev and V.~V. Nikulin, \emph{{Exceptional singularities of {V.I.
  Arnold} and $K3$ surfaces}}, Proc. USSR Topological Conference in Minsk
  (1977).

\bibitem{GKY}
Y.~Goto, R.~Kloosterman, and N.~Yui, \emph{Zeta-functions of certain
  {$K3$}-fibered {C}alabi-{Y}au threefolds}, Internat. J. Math. \textbf{22}
  (2011), no.~1, 67--129.

\bibitem{HS}
B.~Hunt and R.~Schimmrigk, \emph{{$K3$}-fibered {C}alabi--{Y}au threefolds.
  {I}. {T}he twist map}, Internat. J. Math. \textbf{10} (1999), no.~7,
  833--869.

\bibitem{krawitz}
M.~Krawitz, \emph{F{JRW} rings and {L}andau-{G}inzburg mirror symmetry},
  ProQuest LLC, Ann Arbor, MI, 2010, Thesis (Ph.D.)--University of Michigan.

\bibitem{nikulinfactor}
V.~V. Nikulin, \emph{Factor groups of groups of the automorphisms of hyperbolic
  forms with respect to subgroups generated by 2-reflections}, 1979,
  pp.~1156--1158.

\bibitem{nikulinmir}
\bysame, \emph{Integral quadratic forms and some of its geometric
  applications}, 1979, pp.~103--167.

\bibitem{voisin}
C.~Voisin, \emph{Miroirs et involutions sur les surfaces {$K3$}}, Ast\'erisque
  (1993), no.~218, 273--323, Journ{\'e}es de G{\'e}om{\'e}trie Alg{\'e}brique
  d'Orsay (Orsay, 1992).

\bibitem{yasuda}
T.~Yasuda, \emph{Twisted jets, motivic measures and orbifold cohomology},
  Compos. Math. \textbf{140} (2004), no.~2, 396--422.

\end{thebibliography}

\end{document}